\documentclass[12pt, a4paper]{article}

\usepackage{pdfsync}
\usepackage{amsmath}\multlinegap=0pt
\usepackage[latin1]{inputenc}
\usepackage{amsthm}
\usepackage{amssymb}
\usepackage{bbm}
\usepackage{hyperref}
\usepackage{graphicx}
\usepackage{xcolor}
\graphicspath{ {./images/} }
\usepackage{float}
\numberwithin{equation}{section}
\theoremstyle{plain}
\newtheorem{thm}{Theorem}[section]

\newtheorem{prop}[thm]{Proposition}
\newtheorem{lem}[thm]{Lemma}
\newtheorem{defi}[thm]{Definition}

\theoremstyle{definition}

\theoremstyle{remark}
\newtheorem{rem}[thm]{Remark}

\newcommand{\ess}{\text{\rm{ess}}\sup}

\newcommand{\R}{\mathbb{R}}

\newcommand\N{{\mathbb N}}
\newcommand\Z{{\mathbb Z}}

\newcommand\pref[1]{(\ref{#1})}

\let \eps\varepsilon

\DeclareMathOperator{\argmin}{argmin}
\DeclareMathOperator{\spt}{spt}

\DeclareMathOperator{\diam}{diam}

\def\<#1,#2>{\left<#1,#2\right>}
\let\bar\overline

\newcommand\tI{\widetilde {I}}
\newcommand\tx{\widetilde {x}}
\newcommand\ty{\widetilde {y}}

\def\esup{\mathop\mathrm{ess\,sup\,}}
\newcommand{\Prob}{\mathcal P}

\newcommand{\densgamma}{\frac{ \text{d} \gamma }{\text{d} \mu \otimes \nu}}

\title{Entropic approximation of $\infty$-optimal transport problems }

\author{Camilla Brizzi \thanks{Dipartimento di Matematica e Informatica, Universit\`a di Firenze, Viale Morgagni 67/a, 50134 Firenze, Italy, \texttt{camilla.brizzi@unifi.it}}
\and
Guillaume Carlier\thanks{CEREMADE, UMR CNRS 7534, Universit\'e Paris 
Dauphine, PSL, Pl. de Lattre de Tassigny, 75775 Paris Cedex 16, France and INRIA-Paris, MOKAPLAN,
\texttt{carlier@ceremade.dauphine.fr}}
\and
Luigi De Pascale \thanks{Dipartimento di Matematica e Informatica, Universit\`a di Firenze, Viale Morgagni 67/a, 50134 Firenze, Italy,
\texttt{luigi.depacale@unifi.it}}
}

\begin{document}

\maketitle
\begin{abstract}
We propose an entropic approximation approach for optimal transportation problems with a supremal cost. 
We establish $\Gamma$-convergence for suitably chosen parameters for the entropic penalization and that this procedure selects  $\infty$-cyclically monotone plans at the limit. We also present some numerical illustrations performed with Sinkhorn's algorithm. 
\end{abstract}


\medskip

\noindent\textbf{MS Classification:} 49Q22, 65K10.

\section{Introduction}\label{sec-intro}

The usual Monge-Kantorovich optimal transport problem consists, given a transportation cost and distribution of sources and targets, in finding a transport plan making the average transport cost minimal. It has attracted a considerable amount of attention in the last three decades, as can be seen from the textbooks of Villani \cite{Villani1}, \cite{Villani2} and Santambrogio \cite{Santambrogio}.  In optimal transport probems with a supremal cost (also called $L^\infty$ optimal transport), one rather looks for transport plans which minimize the essential supremum of the cost. Whereas the usual Monge-Kantorovich problem is linear programming, $L^\infty$ optimal transport leads to non-convex optimization (eventhough the supremal cost has convex sublevel sets), which to  a large extent, explains why there are less theoretical results and numerical methods (with the notable exception of the recent combinatorial approach  of Bansil and Kitagawa \cite{BK}) to address them. As in the Calculus of Variations with a supremal functional, $L^\infty$ optimal transport may admit many minimizers and selecting special ones which satisfy tractable optimality conditions is an important issue, which was studied first by Champion, De Pascale and Juutinen in \cite{CDPJ}. In contrast with the classical Monge-Kantorovich problem, where restrictions of optimal plans remain optimal between their marginals, this might be false for $L^\infty$ optimal transport. This is why the authors of \cite{CDPJ} have introduced the notion of restrictable optimal and shown that such restrictable solutions are characterized by a remarkable property of $\infty$-cyclical monotonicity of their support. This was the starting point for the existence of optimal maps for $L^\infty$ optimal transport under various conditions on the cost and the marginals, see \cite{CDPJ}, \cite{Jyl}, \cite{BDPK}.

\smallskip

Among numerical methods for optimal transport (see Cuturi and Peyr\'e \cite{CP19}, Benamou \cite{Benamou}, M\'erigot and Thibert \cite{QMBT}),  the entropic penalization approach and the Sinkhorn algorithm have gained a lot of popularity since Cuturi's paper \cite{Cuturi13}.  Entropic optimal transport, which has connections with large deviations and the so-called Schr\"odinger bridge problem, see L\'eonard \cite{Leonard12} has also stimulated an intensive stream of recent theoretical research, see the lecture notes of Nutz \cite{Nutz21notes} and the references therein. A recent breakthrough in the field is  the work of Bernton, Ghosal and Nutz \cite{BGN} where a large deviations principle, related to cyclical monotonicity is established for entropic optimal plans.

\smallskip

The goal of the present paper is to investigate, theoretically and numerically, whether the entropic approximation strategy can be used  for  $L^\infty$ optimal transport as well. We will in particular see how the results of \cite{BGN} can be used to show that this approximation selects at the limit the distinguished restrictable $\infty$-cyclically monotone minimizers introduced in \cite{CDPJ}.

\smallskip

The article is organized as follows. The setting is introduced in Section \ref{sec-prelim}. Section \ref{sec-gamma} is devoted to $\Gamma$-convergence towards the supremal cost functional. In Section \ref{sec-icm}, we study how the entropic penalization selects $\infty$-cyclically monotone plans in the limit. In Section \ref{sec-speed}, we give some quantitative convergence estimates and a large deviations upper bound in the spirit of \cite{BGN}.  Finally, we present  some numerical illustrations in Section \ref{sec-num}. 

\section{Assumptions and notations}\label{sec-prelim}

In the sequel, we will  always assume  that the transportation cost is continuous and nonnegative, $c\in C(\R^d\times\R^d,\R_+)$, and  that the fixed marginals of the problem, $\mu,\nu$ are  two Borel probability measures on $\R^d$, $\mu,\nu\in\Prob(\R^d)$, with compact support. Let $\Pi(\mu, \nu)$ 
be the set of transport plans between $\mu$ and $\nu$ i.e. the set of  Borel probability measures on $\R^d\times \R^d$ having $\mu$ and $\nu$  as marginals. More precisely, a Borel probability measure $\gamma$ on $\R^d \times \R^d$ belongs to $\Pi(\mu, \nu)$ when
\[\gamma(A\times \R^d)=\mu(A) \mbox{ and } \gamma(\R^d\times A)=\nu(A),\]
for every Borel subset $A$ of $\R^d$. Note that every $\gamma$
in $\Pi(\mu, \nu)$ has its support in $\spt(\mu)\times \spt(\nu)$ and that $c$ is uniformly continuous on $\spt(\mu)\times \spt(\nu)$. We are interested in the following $\infty$-optimal transport problem (see
\cite{CDPJ}):
\begin{equation}\label{mainpbm}\tag{$\infty$-OT}
\inf_{\gamma\in \Pi(\mu, \nu)} \gamma-\esup  c=\Vert c \Vert_{L^{\infty}(\gamma)}.
\end{equation}
In contrast with classical optimal transport where one minimizes an integral cost,
\begin{equation}\label{mongekant}\tag{OT}
\inf_{\gamma\in \Pi(\mu, \nu)} \int_{\R^d\times \R^d}c(x,y)d\gamma,
\end{equation} 
\eqref{mainpbm} is a non-convex and presumably  harder problem. \\
Due to the success of entropic approximation of optimal transport with regularization parameter $\eps>0$
\begin{equation}\label{EOT}\tag{$\eps$-EOT}
	\inf_{\gamma\in \Pi(\mu, \nu)}\int_{\R^d\times \R^d}c(x,y)d\gamma + \eps H(\gamma|\mu\otimes\nu), 
\end{equation}
recalled in the introduction, 
it seems natural to introduce, for $\eps>0$ and exponent $p\geq 1$ the following functional $J_{p,\eps}:\Prob(\R^d\times\R^d)\to\R\cup\{+\infty\}$ 
\[J_{p,\eps}(\gamma):=\begin{cases}
\left(\int_{\R^d\times \R^d} c(x,y)^pd\gamma(x,y)+\eps H(\gamma|\mu\otimes\nu) \right)^{\frac{1}{p}} \quad &\text{if} \ \gamma\in\Pi(\mu,\nu),\\
+\infty \quad &\text{otherwise,}
\end{cases}\]
where $H$ stands for relative entropy:
\[H(\gamma|\mu\otimes\nu)=\begin{cases} \int_{\R^d\times \R^d} \log\Big( \frac{ \text{d} \gamma }{\text{d} \mu \otimes \nu}\Big) \mbox{d} \gamma  & \mbox{ if $\gamma \ll \mu\otimes \nu$}, \\ + \infty \quad &  \mbox{ otherwise.} \end{cases}\]
Note that due to strict convexity of the entropy, for every $\eps>0$ and $p\geq 1$, $J_{p, \eps}$ admits a unique minimizer. We now denote by $J_{\infty}:\Prob(\R^d\times\R^d)\to\R\cup\{+\infty\}$, the supremal functional
\begin{equation*}
J_\infty(\gamma):=\begin{cases}
\gamma-\esup  c   \quad &\text{if} \ \gamma\in\Pi(\mu,\nu),\\
+\infty \quad &\text{otherwise}.
\end{cases}
\end{equation*}
Finally, let us set
\[J_{p}:=J_{p,1}.\]
Since $H(\gamma\vert \mu \otimes \nu)\ge 0$ with an equality exactly when $\gamma=\mu \otimes \nu$, $J_{p, \eps}(\gamma) \ge \Vert c\Vert_{L^p(\gamma)}$ but also $\Vert c\Vert_{L^p(\gamma)} \le J_{\infty}(\gamma)$. So, roughly speaking  both approximations play in opposite directions: adding the entropic term is an approximation \emph{from above} but approximating $\Vert c\Vert_{L^{\infty}(\gamma)}$ by $\Vert c\Vert_{L^p(\gamma)}$ is an approximation \emph{from below}. \\We also observe that letting $p\to \infty$ and $\eps\to 0$ is not enough to ensure that minimizers of $J_{p,\eps}$ converge to a minimizer of $J_{\infty}$ (i.e. a solution of \ref{mainpbm}). Indeed, if $\Vert c\Vert_{\infty} \leq \frac{1}{2}$ and $\eps=\frac{1}{p}$ the minimizer $\gamma_p$ of $J_{p, \frac{1}{p}}$ satisfies
\[H(\gamma_p \vert \mu\otimes \nu) \le p 2^{-p}\]
hence $\gamma_p$ converges (actually strongly by Pinsker's inequality, see e.g. Lemma 2.5 in \cite{Tsy}) to $\mu\otimes \nu$ which in general is not a minimizer of $J_{\infty}$. On the one hand, this suggests that $\Gamma$-convergence of the regularizations above to $J_{\infty}$ require conditions relating $\eps$ to $p$. On the other hand, in the previous example, we see that the range of $c^p$ compared to the size of the entropic penalization $\eps$ is crucial. But the solutions of the $\infty$-optimal transport problem are invariant when one replaces $c$ by an increasing function of $c$, in particular one can replace $c$ by $c+2$ (say) so that $c^p$ will typically dominate the entropic term and one can expect $\Gamma$-convergence as $p\to \infty$ for a fixed (or even large) value of $\eps$ (see the next section for more details).

\section{$\Gamma$-convergence}\label{sec-gamma}

Our first result concerns the $\Gamma$-convergence of $J_{p, \eps}$ to $J_\infty$:

\begin{thm}\label{th:gammaconvergence}
Under the general assumptions of Section \ref{sec-prelim} we have:

\begin{enumerate}

\item $J_{p, \eps_p}$ $\Gamma$-converges (for the  weak star topology of  $\Prob(\spt(\mu)\times \spt(\nu)$) to $J_{\infty}$ as $p\to \infty$ provided $\eps_p^{\frac{1}{p}} \to 0$ as $p\to \infty$,

\item if, in addition, $c\ge 1+\lambda$ with $\lambda \geq 0$, then $J_{p, \eps_p}$  $\Gamma$-converges to $J_{\infty}$ as $p\to \infty$ provided 
\begin{equation}\label{condepsp}
\lim_{p \to \infty} \frac{1}{p} \log\Big(1+\eps_p \frac{\log(p)}{(1+\lambda)^p}\Big)=0.
\end{equation}
In particular, in this case, $J_{p,1}$ and $J_{p,p}$ $\Gamma$-converge to $J_{\infty}$ as $p\to \infty$. 
\end{enumerate}
\end{thm}

\begin{proof}
1. Let $\gamma_p\in \Pi(\mu, \nu)$ converge weakly star to $\gamma$. By nonnegativity of $H(\gamma_p \vert \mu \otimes\nu)$, we have
\[\liminf_p J_{p, \eps_p} (\gamma_p)\ge \liminf_p \Vert c\Vert_{L^p(\gamma_p)}.\]
Hence, for fixed $q$, since $\Vert c\Vert_{L^p(\gamma_p)} \ge \Vert c\Vert_{L^q(\gamma_p)}$ for $p\ge q$, we have 
\[\liminf_p J_{p, \eps_p} (\gamma_p)\ge \liminf_p \Vert c\Vert_{L^q(\gamma_p)} =\Vert c\Vert_{L^q(\gamma)}\]
taking the supremum with respect to $q$ thus yields the desired $\Gamma$-liminf inequality
\[\liminf_p J_{p, \eps_p} (\gamma_p)\ge \Vert c\Vert_{L^{\infty}(\gamma)}=J_{\infty}(\gamma).\]
Let us now prove the $\Gamma$-limsup inequality. For any $\gamma \in \Pi(\mu, \nu)$ we consider $\gamma^\delta$, the block approximation of $\gamma$ at scale $\delta\in (0,1)$ defined by \eqref{eq:blockapprox} below, whose convergence to $\gamma$ is guaranteed by the first inequality in \eqref{errorblock}. By concavity, we first have for $p\geq 1$, 
\[\begin{split}
J_{p, \eps_p}(\gamma^\delta) & \le \Vert c \Vert_{L^p(\gamma^\delta)}+ \eps_p^{\frac{1}{p}} H(\gamma^\delta \vert \mu \otimes \nu)^{\frac{1}{p}}\\
& \le \Vert c \Vert_{L^{\infty}(\gamma^\delta)}+ \eps_p^{\frac{1}{p}} H(\gamma^\delta \vert \mu \otimes \nu)^{\frac{1}{p}.} 
\end{split}\]
Denoting by $\omega$ a modulus of continuity of $c$ on $\spt(\mu)\times \spt(\nu)$, thanks to the first inequality in \eqref{errorblock}, we have
\[ \Vert c \Vert_{L^{\infty}(\gamma^\delta)} \leq \Vert c \Vert_{L^{\infty}(\gamma)}+ \omega(\sqrt{2d}\delta),\]
being $\sqrt{2d}\delta$ the diameter of the cubes of the approximation.
Moreover, by the second inequality in \pref{errorblock}, we have 
\[  H(\gamma^\delta \vert \mu \otimes \nu)^{\frac{1}{p}} \leq d^{\frac{1}{p}} \log(L/\delta)^{\frac{1}{p}}   \]
so if we define $\gamma_p$ as the  block approximation of $\gamma$ at scale $\delta=\frac{1}{p}$ (say), we obtain
\[ \limsup_p J_{p, \eps_p} (\gamma_p) \le J_{\infty}(\gamma)+ \limsup_p \Big( \omega\Big(\frac{\sqrt{2d}}{p}\Big)+  d^{\frac{1}{p}}\eps_p^{\frac{1}{p}} \log(Lp)^{\frac{1}{p}}  \Big)=J_\infty (\gamma),  \]
since we have assumed that $\eps_p^{\frac{1}{p}} \to 0$ as $p\to +\infty$. 

\smallskip

2. Let us now assume that $c\geq 1+\lambda$, the proof of the $\Gamma$-liminf inequality for $J_{p, \eps_p}$ is exactly as above. For $\gamma \in \Pi(\mu, \nu)$ and $\gamma_p$ the 
 block approximation of $\gamma$ at scale $\frac{1}{p}$, we have
 \begin{align} 
 \notag J_{p, \eps_p} (\gamma_p)  & \le \Vert c \Vert_{L^{\infty}(\gamma_p)} \Big(1+  \frac{d  \eps_p\log(L p)}{(1+\lambda)^p}  \Big)^{\frac{1}{p}}\\
\label{estimatecbig} &\le \Big(J_{\infty}(\gamma)+\omega\Big(\frac{\sqrt{2d}}{p}\Big)\Big)  \Big(1+  \frac{d  \eps_p\log(L p)}{(1+\lambda)^p } \Big)^{\frac{1}{p}}
\end{align}
so that, as soon as \pref{condepsp} holds, one has 
\[\limsup_p J_{p, \eps_p} (\gamma_p) \le J_{\infty}(\gamma).\]

\end{proof}

\begin{rem}
Notice that in case $c\ge1+\lambda$ for some $\lambda>0$, $\Gamma$-convergence of $J_{p,\eps_p}$ to $J_\infty$ is guaranteed even  for fastly increasing $\eps_p$  like $\eps_p=p^m(1+\lambda)^p$ with $m\geq 0$.  On the contrary, in the general case,  the condition $\eps_p^{\frac{1}{p}} \to 0$ requires to choose values of $\eps$ way too small to be used in practice for numerical computations. This suggests in practice to rescale the cost  so that it is bounded from below by $1$. 

\end{rem}
\begin{rem} \label{rem:weakerassgammaconv}
We observe that in \eqref{estimatecbig} it is sufficient that $||c||_{L^\infty(\gamma_p)}\ge 1+\lambda$, therefore the conclusion of case 2. in Theorem \ref{th:gammaconvergence} remains valid under the weaker  assumption that  $v_\infty=\min_{\Pi(\mu, \nu)}  J_\infty\ge 1+\lambda$. 
\end{rem}
For the $\Gamma$-limsup inequality, we have used the block approximation introduced in \cite{CDPS}, which is defined as follows:
 \begin{defi}\label{def:blockapprox}
Let $\gamma\in \Pi(\mu, \nu)$. For $\delta>0$ and $k\in \Z^d$, we denote by $Q_k^{\delta}$ the cube $\delta (k+[0,1)^d)$. The block approximation of $\gamma$ at scale $\delta\in (0,1)$
is then defined by
\begin{equation}\label{eq:blockapprox}
	\gamma^{\delta} := \sum_{k, l \in \Z^d \; : \; \mu(Q_k^\delta)>0, \, \nu(Q_l^\delta)>0} {\gamma(Q_k^\delta \times Q_l^\delta)} \mu_k^\delta \otimes \nu_l^\delta
\end{equation}
where $\mu_k^\delta$ and $\nu_l^\delta$  are defined by
\[\mu_k^\delta(A)=\frac{\mu(Q_k^\delta \cap A)}{\mu(Q_k^\delta)}, \; \nu_l^\delta(A)=\frac{\nu(Q_l^\delta \cap A)}{\nu(Q_l^\delta)}\]
for every Borel subset $A$ of $\R^d$. 
\end{defi}

For the sake of completeness, we give a short proof of the properties of the block approximation that we have used in the proof of Theorem \ref{th:gammaconvergence} (see \cite{CDPS} and \cite{CPT} for related results):

\begin{lem}\label{blockapprox}
Let  $\gamma\in \Pi(\mu, \nu)$ and $\gamma^{\delta}$ be the block approximation of $\gamma$ at scale $\delta\in (0,1)$, then  $\gamma^\delta\in \Pi(\mu, \nu)$ and 
\begin{equation}\label{errorblock} 
W_{\infty}(\gamma^\delta, \gamma) \leq \sqrt{2d} \delta, \; H(\gamma^{\delta} \vert \mu \otimes \nu) \leq d \log\Big(\frac{L}{\delta}\Big),
\end{equation}
where $L$ is a constant depending only on $\spt(\mu)$ (actually on its diameter). 
\end{lem}
\begin{proof}
The fact that $\gamma^\delta\in \Pi(\mu, \nu)$ is easy to check by construction (see  \cite{CDPS}).  Now observe that by \eqref{eq:blockapprox} the density of $\gamma^{\delta}$ with respect to $\mu\otimes\nu$ is 
\[\frac{d\gamma^\delta}{d\mu\otimes\nu}(x,y)=\begin{cases}
	 \frac{\gamma(Q_k^\delta \times Q_l^\delta)}{\mu(Q_k^\delta) \nu(Q_l^\delta)} \quad &\text{if }(x,y)\in Q_k^\delta\times Q_l^\delta, \ \text{and }\mu(Q_k^\delta), \, \nu(Q_j^\delta)>0,\\
	 0 \quad &\text{otherwise}.
\end{cases}\]
Therefore
\[ \begin{split}H(\gamma^\delta\vert\mu\otimes\nu)&= 
 \sum_{k, l \in \Z^d \; : \; \mu(Q_k^\delta)>0, \, \nu(Q_l^\delta)>0}\int_{Q_k^\delta\times Q_l^\delta}\log\left(\frac{\gamma(Q_k^\delta \times Q_l^\delta)}{\mu(Q_k^\delta) \nu(Q_l^\delta)} \right)d\gamma^\delta \\
 &\le\sum_{k, l \in \Z^d \; : \; \mu(Q_k^\delta)>0, \, \nu(Q_l^\delta)>0}\int_{Q_k^\delta\times Q_l^\delta}\log\left(\frac{1}{\mu(Q_k^\delta)} \right)d\gamma^\delta\\
 &=\sum_{k \in \Z^d \; : \; \mu(Q_k^\delta)>0}\mu(Q_k^\delta)\log\left(\frac{1}{\mu(Q_k^\delta)}\right),
\end{split} \]
where the inequality is due to the fact that \(\frac{\gamma(Q_k^\delta \times Q_l^\delta)}{\nu(Q_l^\delta)}\le 1\), while the last equality is obtained summing over $\l$. If $L\geq 1$ is such that $\spt \mu$ is contained in a cube of side $L-1$, the number of cubes $Q_k^\delta$ with positive $\mu$-measure is not greater than $N_\delta:= \left(\frac{L}{\delta}\right)^d$. Therefore, applying Jensen's inequality to the concave function $f(z)=z\log(\frac{1}{z})$, we have
\[\begin{split}
H(\gamma^\delta\vert\mu\otimes\nu)&\le\sum_{k=1}^{N_\delta}\mu(Q_k^\delta)\log\left(\frac{1}{\mu(Q_k^\delta)}\right)\\
&\le N_\delta\left(\frac{1}{N_\delta}\sum_{k=1}^{N_\delta}\mu(Q_k^\delta)\log\left(\frac{1}{\sum_{k=1}^{N_\delta}\frac{1}{N_\delta}\mu(Q_k^\delta)}  \right)\right)\\
&=\log(N_\delta)=d\log(L)-d\log(\delta),
\end{split}\]
which proves the second inequality in \eqref{errorblock}.

\smallskip

 By construction $\gamma(Q_k^\delta\times Q_l^\delta)=\gamma^\delta(Q_k^\delta\times Q_l^\delta)$, for any $k,l$. Let   $J$  be the set of pairs of indices  $(k,l)$ such that $\gamma^\delta(Q_k^\delta\times Q_l^\delta)>0$ and set $\bar{Q}_{j}=Q_k^\delta\times Q_l^\delta$, for any $j=(k,l)\in J$. We define
\[\eta^\delta:=\sum_{j \, : \, \gamma (\bar{Q}_{j})>0}\gamma(\bar Q_j) \gamma_j\otimes\gamma_j^\delta, \]
where $\gamma_j(A):=\frac{\gamma(A\cap \bar{Q}_{j})}{\gamma(\bar Q_j)}$ and $\gamma^\delta_j(A):=\frac{\gamma^\delta(A\cap \bar{Q}_{j})}{\gamma^\delta(\bar Q_j)}$. By construction $\eta^\delta\in\Pi(\gamma,\gamma^\delta)$, thus
\[W_{\infty}(\gamma,\gamma^\delta)\le\vert\vert x-y \vert\vert_{L^{\infty}(\eta^\delta)}\le \diam(\bar Q_j)=\sqrt{2d}\delta.\]
\end{proof}

\section{Selection of plans with $\infty$-cyclically monotone support}\label{sec-icm}

As shown in \cite{CDPJ} and \cite{Jyl}, restrictable minimizers of $J_\infty$ are supported on  $\infty$-cyclically monotone sets,  such sets are defined as follows:

	\begin{defi}
	A set $\Gamma\subset \R^d\times \R^d$ is said to be $\infty$-cyclically monotone if  we have that 
	\[\max_{i=1,\dots,k}\left\{c(x_{i},y_{i})\right\} \le \max_{i=1,\dots,k}\left\{c(x_{i},y_{i+1})\right\},\] 
	for all $k\in\N^*$ and $\left\{(x_i,y_i) \right\}_{i=1}^{k}\subset \Gamma$, where $y_{k+1}=y_1$.
	A transport plan $\gamma$ is said to be $\infty$-cyclically monotone if $\spt\gamma$ is an $\infty$-cyclically monotone set.  
\end{defi}
Since every permutation can be obtained as composition of cycles on disjoint sets and trivial cycles on fixed points, one can see that $\infty$-cyclical monotonicity of a set $\Gamma\subset \R^d\times \R^d$ is equivalent to the fact that for every $k\in\N^*$, every  $\left\{(x_i,y_i) \right\}_{i=1}^{k}\subset \Gamma$ and every $\sigma\in \Sigma(k)$ (where $\Sigma(k)$ is the permutation group of $\{1, \ldots, k\}$), one has
	\[\max_{i=1,\dots,k}\left\{c(x_{i},y_{i})\right\} \le \max_{i=1,\dots,k}\left\{c(x_{i},y_{\sigma(i)})\right\}.\] 
Usually, in the literature, the previous definition is called $\infty$-$c$-cylical monotonicity, to keep notations simple, we have omitted the dependence on the cost $c$;  let us remark that $\infty$-cyclical monotonicity is invariant by replacing $c$ by a strictly increasing transformations of $c$ (like $c^p$ with $p>0$), contrarily to the usual notion of $c$-cyclical monotonicity. We recall that a nonempty subset $\Gamma$  of $\R^d \times \R^d$ is called $c$-cyclically monotone when for every $k\in \N^*$, every $(x_i,y_i)_{i=1}^k\subset\Gamma$ and every permutation $\sigma \in \Sigma(k)$, one has
	\begin{equation}\label{ccmon}
		\sum_{i=1}^k c(x_i, y_i) \leq \sum_{i=1}^k c(x_i, y_{\sigma(i)}).
\end{equation} 
Our goal in this section is to investigate the convergence of the entropic approximation to $\infty$-cyclically monotone plans. We shall make use of the analysis of the landmark recent article \cite{BGN}. Let us first recall the notion of \((c,\eps)\)-cyclically invariance introduced in \cite{BGN}:
\begin{defi}\label{def:cyclinvariance}
	Let $c:\R^d\times\R^d\to(0,\infty)$ be a measurable function. A coupling $\gamma\in\Pi(\mu,\nu)$ is called $(c,\eps)$\textit{-cyclically invariant} if $\gamma\ll\mu\otimes\nu$ and its density admits a representative $\densgamma:\R^d\times\R^d\to(0,\infty)$ such that 
\begin{multline*}
\prod_{i=1}^{k}\densgamma(x_i,y_i)=\\ \exp\left(-\frac1\eps\left[\sum_{i=k}^{k}(c(x_i,y_i)-c(x_i,y_{i+1}))\right] \right)\prod_{i=1}^{k}\densgamma(x_i,y_{i+1}),
\end{multline*}
	for all $k\in\N^*$ and \( \left\{(x_i,y_i) \right\}_{i=1}^{k}\subset \R^d\times\R^d\), where \(y_{k+1}=y_1\).
\end{defi} 
In \cite{BGN} (Proposition 2.2), it is shown that whenever \eqref{EOT}
is finite, the (unique) solution $\gamma_{\eps}$ of \eqref{EOT} 
is characterized by being \((c,\eps)\)-cyclically invariant. The next lemma, which is a part of Lemma 3.1 in \cite{BGN}, provides an estimate for $(c,\eps)$-cyclically invariant couplings, which will be useful for our purpose. For the reader's convenience we provide also here the proof. 
\begin{lem}\label{nutzlemma}
Let $\eps>0$ and $\gamma_{\eps}\in\Pi(\mu,\nu)$ be $(c,\eps)$-cyclical invariant. For every fixed $k\ge 2$, $k\in\N$, and $\delta\ge0$, let $A_{k,c}(\delta)$ be the set defined by
\begin{equation}\label{Adelta}
A_{k, c} (\delta):=\left\{\left(x_i,y_i\right)_{i=1}^k\in \left(\R^d\times\R^d\right)^k \, : \, \sum_{i=1}^{k}c(x_i,y_i)-\sum_{i=1}^{k}c(x_i,y_{i+1})\ge \delta \right\}
\end{equation}
where \(y_{k+1}=y_1\). Let $A\subset A_{k, c}(\delta)$ be Borel. Then $\gamma_\eps^k:=\prod_{i=1}^k(\gamma_\eps)(\text{d}x_i,\text{d}y_i)$ satisfies
\[\gamma_\eps^k(A)\le e^{\frac{-\delta}{\eps}}.\]
\end{lem}

\begin{proof}
By Definition \ref{def:cyclinvariance} of $(c,\eps)$-cyclical invariance, for $\gamma_\eps^k$ a.e. $(x_i,y_i)_{i=1}^k\in A$ we have that
	\[\prod_{i=1}^{k}\frac{d\gamma_\eps}{\mu\otimes\nu}(x_i,y_i)\le e^{-\frac\delta\eps}\prod_{i=1}^{k}\frac{d\gamma_\eps}{\mu\otimes\nu}(x_i,y_{i+1}). \] 
In one defines the set $\bar A:=\{(x_i,y_{i+1}) \, : \, (x_i,y_{i})\in A  \}$, by integrating over $A$ with respect to $\gamma_\eps^k=\prod\gamma_\eps(x_i,y_i)=\prod\gamma_\eps(x_i,y_{i+1})$ we obtain 
\[\gamma_\eps^k(A)\le e^{-\frac\delta\eps}\gamma_\eps^k(\bar A)\le e^{-\frac\delta\eps}.\]

\end{proof}

The fact that the entropic approximation procedure selects $\infty$-cyclically monotone plans is then ensured by the following:

\begin{thm}\label{th:inftymon}
	Under the general assumptions of Section \ref{sec-prelim}, further assume that $c>0$ everywhere, and let $\gamma_{p, \eps_p}$ be the minimizer of $J_{p, \eps_p}$.  Then, any weak star cluster point \(\gamma_{\infty} \) as \(p\to\infty\) of the family $\{\gamma_{p, \eps_p}\}_{p\ge 1} $ is $\infty$-cyclically monotone, provided
	\begin{enumerate}
	\item  $\eps_p^{\frac{1}{p}} \to 0$ as $p\to \infty$,
	\item $\eps_p=o(p(1+\lambda)^p)$ if, in addition, $c\geq 1+ \lambda$ with $\lambda\geq 0$. 
	
\end{enumerate}	
		
	\end{thm}

 \begin{proof}
Up to extracting a subsequence, let us assume  that $\gamma_{p, \eps_p}$ weakly star converges to  $\gamma_\infty$. We proceed by contradiction assuming that there exist $\delta>0$ and a finite sequence of points $\left(x_i,y_i\right)_{i=1}^k$ contained in $\spt\gamma_\infty$, such that \[\max_{i=1,\dots,k}\left\{c(x_i,y_i)\right\}>\max_{i=1,\dots,k}\left\{c(x_i,y_{i+1})\right\}+\delta. \]
By the continuity of the cost function $c$ and by the uniform convergence of  $\left(\sum_{i=1}^{k}c(x'_i,y'_i)^p\right)^{\frac1p}$ to $\max_{i=1,\dots,k}\{c(x'_i,y'_i)\}$, as $p\to+\infty$, we deduce that for every $i=1,\dots,k$ there exists an open neighborhood $U_i$ of $(x_i,y_i)$ and $p(\delta)>0$, such that 
\[\left(\sum_{i=1}^{k}c(x'_i,y'_i)^p\right)^{\frac1p}>\left(\sum_{i=1}^{k}c(x'_i,y'_{i+1})^p\right)^{\frac1p}+\delta,\]
for every $(x'_i,y'_i)\in U_i$ (again with the convention that $y'_{k+1}=y'_1$) and $p\ge p(\delta)$. We now observe that 
\begin{align}
\notag&\sum_{i=1}^{k}c(x'_i,y'_i)^p>\left(\left(\sum_{i=1}^{k}c(x'_i,y'_{i+1})^p\right)^{\frac1p}+\delta\right)^p\\&\ge \sum_{i=1}^{k}c(x'_i,y'_{i+1})^p+p\left(\sum_{i=1}^{k}c(x'_i,y'_{i+1})^p\right)^\frac{p-1}{p}\delta,\label{pdelta}
\end{align}
where the last inequality follows from the convexity of $t\mapsto t^p$, with $p>1$. Since $c>0$ there exists some $b>0$ such that $c\geq b$ on each $U_i$, $i=1, \dots, k$, hence, for every $(x'_i,y'_i)\in U_i$ and $p\ge p(\delta)$
\begin{equation}\label{pdeltab}
\sum_{i=1}^{k}c(x'_i,y'_i)^p>\sum_{i=1}^{k}c(x'_i,y'_{i+1})^p+p\delta b^{p-1}.
\end{equation}
We thus have  $U_1\times\dots\times U_k\subset A_{k, c^p}(p\delta b^{p-1})$, where $A_{k, c^p}(p\delta b^{p-1})$ is defined as in \eqref{Adelta} with $c$ replaced by $c^p$.
Applying Lemma \ref{nutzlemma}, we thus get:
\begin{align}
	\notag&\gamma_{\infty}^k(U_1\times\dots\times U_k):=\prod_{i=1}^k    \gamma_{\infty} (U_i) \\
	\notag&\le\liminf_p\gamma^k_{p,\eps_p}(U_1\times\dots\times U_k):=\prod_{i=1}^k \gamma_{p,\eps_p}  (U_i) \\&
	\le \liminf_p e^{-\frac{p\delta b^{p-1}}{\eps_p}}\label{estimate}
\end{align}  
so that if $\eps_p^{\frac{1}{p}} \to 0$ as $p\to \infty$, for large enough $p$ one has $\eps_p \leq b^p$, which yields
\[\liminf_p e^{-\frac{p\delta b^{p-1}}{\eps_p}}=0.\]
On the other hand, since the points $(x_i,y_i)$ belong to $\spt \gamma_\infty$, we have that  $\gamma_{\infty}^k(U_1\times\dots\times U_k)>0$,  which yields the desired contradiction. This shows the first assertion. Now, if $c\geq (1+\lambda)$  with $\lambda \geq 0$, we can replace $b$ by $(1+\lambda)$ in \pref{estimate} and the same conclusion will be reached as soon as $\eps_p=o(p(1+\lambda)^p)$, proving the second assertion. 


\end{proof}

\begin{rem} 
Despite what we observed in Remark \ref{rem:weakerassgammaconv} regarding Theorem \ref{th:gammaconvergence}, in the proof of the second assertion of Theorem \ref{th:inftymon}, it does not seem that the condition $c(x,y)\ge1$ for every $(x,y)$ can be weakened to $J_\infty \geq 1$. Note also that the condition  $\eps_p=o(p(1+\lambda)^p)$ is stronger than condition  \pref{condepsp}  that guarantees $\Gamma$-convergence when $c\geq 1+\lambda$. 
\end{rem}

\section{Some estimates on the speed of convergence}\label{sec-speed}

Our aim in this Section is to give some error estimates for $v_p-v_\infty$ where
\begin{equation}\label{defvpvinf}
v_p:=\min_{\gamma\in\Pi(\mu,\nu)}J_p \quad \text{and} \quad v_{\infty}:=\min_{\gamma\in\Pi(\mu,\nu)}J_{\infty},
\end{equation}
where $J_p:=J_{p,1}$ (i.e. for the sake of simplicity we take $\eps_p=1$ as entropic penalization parameter).

\subsection{Upper  bounds}

\begin{prop}[Upper bounds on the speed of convergence]\label{prop:upperbound}
Let $c\in C^{0,\alpha}(\R^d\times\R^d)$, with $\alpha\in (0,1]$ and let us assume that $v_\infty\ge1+\lambda$ for some $\lambda \geq 0$. 
Then we have
\[ v_{p}-v_{\infty}\le    \begin{cases}   O(e^{-\beta p}), \mbox{ with }  \beta=\min\{\alpha, \log(1+\lambda)\} &\mbox{ if $\lambda>0$}\\
O \Big(  \frac{ \log (\log p))}{p} \Big) &\mbox{ if $\lambda=0$}. \end{cases}\]
\end{prop}
\begin{proof}
Let $\gamma_{\infty}$ be a minimizer of $J_{\infty}$ and $\gamma^{\delta}$ be the block approximation of $\gamma_{\infty}$ at scale $\delta\in (0,1)$, as defined in \eqref{eq:blockapprox}. We observe that, by construction and by the H\"{o}lder condition on $c$,  denoting by $A$ the $C^{0, \alpha}$ semi-norm of $c$, we first have
 \[||c||_{L^{\infty}(\gamma^{\delta})}\le ||c||_{L^{\infty}(\gamma_{\infty})} + A \delta^{\alpha}.\]
Then 
\begin{align}
\notag&v_p\le \left(\int c^pd\gamma^{\delta}+ H(\gamma^{\delta}\vert\mu \otimes\nu) \right)^{\frac1p}\le \left( ||c||^p_{L^{\infty}(\gamma^{\delta})}+ H(\gamma^{\delta}\vert\mu \otimes\nu) \right)^{\frac1p}\\\notag&\le \left(||c||_{L^{\infty}(\gamma_{\infty})} + A \delta^{\alpha}\right)\left( 1+\frac{H(\gamma^{\delta}\vert\mu \otimes\nu)}{\left(1+\lambda\right)^{p}}    \right)^{\frac1p}\\\label{estimatedelta}&\le \left( v_{\infty} + A \delta^{\alpha}\right)\left( 1+\frac{d\log(L/\delta)}{\left(1+\lambda\right)^{p}} \right)^{\frac1p},
\end{align}
where the last inequality follows from Lemma \ref{blockapprox}. 
For $\lambda>0$, choosing $\delta:=e^{-p}$,  \eqref{estimatedelta} becomes (setting $C=d \log(L)$)
\[v_p\le \left(v_{\infty}+A e^{-\alpha p}\right)\left(1+ \frac{C +dp}{(1+\lambda)^p}  \right)^{\frac1p}, \]
then, we observe that for large $p$, one has 
\[\left(1+ \frac{C+dp}{(1+\lambda)^p}  \right)^{\frac1p} =1+ \frac{d}{(1+\lambda)^p} +o\Big(\frac{1}{(1+\lambda)^p}\Big).\]
Therefore, for  $p$ large enough, 
\[v_p\le v_{\infty}+Be^{-\beta p},\]
for some $B>0$ and $\beta=\min\{\alpha, \log(1+\lambda)\}$.

\smallskip

Now if $\lambda=0$, we choose $\delta=p^{-1/\alpha}$ in \eqref{estimatedelta} which gives
\[\begin{split}
 v_p  &\leq \Big(v_\infty + \frac{A}{p}\Big) \exp\Big(\frac{1}{p} \log(1+  d \log (Lp^{1/\alpha}))\Big)\\
& =v_{\infty} + \frac1\alpha\frac{v_\infty}{p} \log(\log(p)) +o\Big( \frac{ \log(\log(p)}{p}\Big) 
 \end{split} \] 
which ends the proof.
\end{proof}

\subsection{Upper and lower bounds in the discrete case}

Let us now consider the discrete case where  there exist $x_1,\dots, x_N$ and $y_1,\dots,y_M$ points in $\R^d$ such that 
\begin{equation}\label{discretemarg}
\mu= \sum_{i=1}^N\mu_{i}\delta_{x_i} \quad \text{and} \quad \nu=\sum_{j=1}^M\nu_j\delta_{y_j}
\end{equation}
with (strictly, without loss of generality) positive weights $\mu_i$ and $\nu_j$ summing to $1$. To shorten notations let us set $c_{ij}=c(x_i, y_j)\geq 0$. In this setting, transport plans $\gamma$ will simply be denoted as $N\times M$ matrices with entries $\gamma^{ij}$. We also recall that in the discrete setting $\Pi(\mu,\nu)$ is a convex polytope and the constraint $\gamma\in\Pi(\mu,\nu)$ is equivalent to 
\[\gamma\mathbbm{1}_M= \left(\sum_{j=1}^M\gamma^{ij}\right)_i=(\mu_i)_i \ \text{and} \ \gamma^\intercal\mathbbm{1}_N=\left(\sum_{i=1}^N\gamma^{ij}\right)_i=(\nu_j)_j.\]

In the discrete setting transport plans have a finite entropy with respect to $\mu \otimes \nu$, with the (crude) bound
\[H(\gamma \vert \mu \otimes \nu) \leq M:=-\sum_{i=1}^N \mu_i \log(\mu_i)-\sum_{j=1}^N \nu_j \log(\nu_j)\]
for every $\gamma \in \Pi(\mu, \nu)$. So if $v_\infty \geq 1+\lambda$ with $\lambda\geq 0$, taking $\gamma_\infty$ a minimizer of $J_\infty$, we obtain
\[\begin{split}
v_p  &\leq J_p(\gamma_\infty)\leq  v_{\infty} \Big(1+ \frac{M}{(1+\lambda)^p} \Big)^{\frac{1}{p}}\\
&\leq v_\infty \Big(1+ \frac{M}{p(1+\lambda)^p} +o\Big(  \frac{M}{p(1+\lambda)^p}   \Big)    \Big)
\end{split}\]
which gives (in a straightforward way, i.e. without using block approximation) an exponentially decaying upper bound for $v_p-v_\infty$ for $\lambda>0$ and an algebraic upper bound $v_p-v_\infty \leq O(1/p)$ if $\lambda =0$. The fact that $v_\infty \geq 1$ therefore ensures that $p(v_p-v_\infty)$ is bounded from above. It turns out, that in the discrete setting, this condition also guarantees that we also have an algebraically decaying  lower bound for the error. To see this, we first need the following:
\begin{lem}\label{discretelb} 
Let $\mu$ and $\nu$ be discrete measures i.e. of the form \pref{discretemarg} and define 
\[F_\infty:=\{\gamma \in \Pi(\mu, \nu) \; : \;  J_{\infty}(\gamma)=v_{\infty}\}\]
and for every $\gamma\in F_{\infty}$, 
\[m(\gamma):=\max \{ \gamma^{ij} \;  : \; \gamma^{ij}>0, \; c_{ij}= v_{\infty}\} \]
then there is some $\theta >0$ such that $m(\gamma)\geq \theta$, for every $\gamma\in F_{\infty}$. 
\end{lem}
\begin{proof}
Since $v_{\infty}$ is the minimum of $J_{\infty}$ over $\Pi(\mu, \nu)$, one can write $F_{\infty}$ as the set of transport plans for which
\[\gamma^{ij}>0 \Rightarrow c_{ij}-v_\infty \leq 0\]
or equivalently
\[l(\gamma):= \sum_{ij} \gamma^{ij} (c_{ij}-v_{\infty})_+ =0.\] 
In other words, $F_\infty$ is the facet of $\Pi(\mu, \nu)$ where the linear form $l$ (which is nonnegative on $\Pi(\mu, \nu)$) achieves its minimum and it is therefore a convex polytope, whose extreme points belong to the (finite) set of extreme points of $\Pi(\mu, \nu)$. Let us then denote by $\{\gamma_a, \, a \in A\}$ with $A$ a finite index set the set of extreme points of $F_{\infty}$.  Thanks to Minkowski's theorem, we can write any $\gamma\in F_{\infty}$ as 
\[\gamma:=\sum_{a\in A} \alpha_a \gamma_a,\]
for some weights $\alpha_a\geq 0$ summing to $1$. In particular we may pick   $a_0 \in A$ with $\alpha_{a_0} \geq \frac{1}{ \vert A\vert}$ (with $\vert A\vert$ denoting the cardinality of $A$). Then we have
\[m(\gamma) \geq \frac{m(\gamma_{a_0})}{\vert A\vert} \geq \theta:= \min_{a \in A} \frac{m(\gamma_a)}{\vert A\vert}>0, \]
where the strict positivity of $\theta$ then follows from the fact that $A$ is finite and $m(\gamma_a )>0$ for every $a\in A$. 
\end{proof}
We are now ready to prove the announced lower bound.

\begin{prop}[Lower bound on the speed of convergence, discrete case]\label{prop:lowerbound}
Assume that $\mu$ and $\nu$ are discrete measures i.e. of the form \pref{discretemarg} and that $v_\infty\geq 1$, then $p(v_p-v_\infty)$ is bounded from below. Hence
\[ v_p-v_\infty =O\Big(\frac{1}{p}\Big).\]


\end{prop}
\begin{proof}
Let us argue by contradiction and assume that $p(v_p-v_\infty)$ is unbounded from below, then there is a sequence $p_n \to \infty$ as $n \to \infty$ such that
\begin{equation}\label{contradlimit}
\lim_n p_n(v_{p_n}-v_\infty)=-\infty.
\end{equation}
Letting $\gamma_n$ be the minimizer of $J_{p_n}$, passing to a subsequence if necessary, we may assume that $\gamma_n$ converges to some $\gamma_\infty$ which belongs to $F_\infty$  (as defined in Lemma \ref{discretelb}) since $v_\infty \geq 1$. In particular, there exists $i_{0}, j_{0}$ such that 
\[ c_{i_0  j_0}=v_{\infty} \mbox{ and } \gamma_{\infty}^{i_0 j_0} \geq \theta>0,\]
where $\theta$ is the lower bound from Lemma \ref{discretelb}. Since $\gamma_n^{i_0 j_0 }$ converges to $\gamma_{\infty}^{i_0 j_0}$ we have, for large enough $n$, $\gamma_n^{i_0 j_0 } \geq \frac{\theta}{2}$ , hence,  using the fact that $c_{i_0  j_0}=v_{\infty}$ and again the nonnegativity of the entropy
\[\begin{split}
v_{p_n}  & \geq v_{\infty} \Big( \frac{\theta}{2} \Big)^{\frac{1}{p_n}}=v_{\infty} \exp \Big( \frac{1}{p_n} \log \frac{\theta}{2}   \Big)\\
&\ge v_{\infty} \Big(1+  \frac{1}{p_n} \log \frac{\theta}{2}  \Big)  \end{split} \]
which is the desired contradiction to \pref{contradlimit}. 

\end{proof}

\subsection{A large deviations upper bound}\label{ldub}
In this (somehow independent) paragraph, our goal is to discuss a (partial) extension of the large deviations results of \cite{BGN} to the $L^\infty$-optimal transport framework. Considering the Monge-Kantorovich problem \eqref{mongekant}
it is well-known (see \cite{GangboM}, \cite{Santambrogio}) that the optimality  for \eqref{mongekant} 
of a plan $\gamma\in \Pi(\mu, \nu)$ is characterized by a property of $c$-cyclical monotonicity of its support $\Gamma:=\spt(\gamma)$, where $c$-cyclical monotonicity is defined by \eqref{ccmon}.
To analyze fine convergence properties of the entropic approximation of \eqref{mongekant}, defined by  \eqref{EOT}, assuming convergence (taking a subsequence if necessary) as $\eps\to 0^+$, of the minimizer $\gamma_\eps$ of \eqref{EOT}
to some $\gamma$ and denoting by $\Gamma$ the $c$-cyclically monotone set $\spt(\gamma)$, the authors of \cite{BGN} introduced 
\[I(x,y):=\sup_{k\ge2}\sup_{(x_i,y_i)_{i=2}^k\subset\Gamma}\sup_{\sigma\in\Sigma(k)} \Big\{\sum_{i=1}^{k}c(x_i,y_i)-\sum_{i=1}^kc(x_i,y_{\sigma(i)}) \Big\}, \; (x,y)\in \R^d\times \R^d\]
with $(x_1, y_1)=(x,y)$.
They  proved that $I$ is a good rate function for the family of optimal entropic plans, $\{\gamma_\eps\}_{\eps>0}$ in the sense that it obeys, under very general conditions, the large deviations principle 
	\[\begin{split}
&\limsup_{\eps\to 0}\eps\log(\gamma_\eps(C))\le -\inf_{(x,y)\in C} I(x,y) \quad \text{and}\\
&\liminf_{\eps\to 0}\eps\log(\gamma_\eps(U))\ge -\inf_{(x,y)\in U}I(x,y),
	\end{split} \]
for every  compact $C$ and  every open $U$ included in $\spt(\mu)\times \spt(\nu)$. Denoting by $\gamma_{p, \eps}$ the minimizer of $J_{p, \eps}$, the results of \cite{BGN} (using $c^p$ instead of $c$) of course apply to the convergence of $\gamma_{p, \eps}$ as $\eps \to 0^+$ for a \emph{fixed} exponent $p$. For $L^\infty$ optimal transport, it makes more sense to rather consider the situation where $\eps>0$ is fixed and $p$ tends to $\infty$. More precisely, we know from Theorem \ref{th:inftymon}, that if $c\geq 1$,  $\eps>0$ is fixed, the family $\{\gamma_{p, \eps}\}_{p\geq 1}$ weakly star converges (again possibly after an extraction) to some $\gamma_{\infty}$ as $p\to \infty$, $\Gamma_\infty:=\spt( \gamma_{\infty})$ is $\infty$-cyclically monotone. 
\smallskip
In addition to the general assumptions of Section \ref{sec-prelim}, we shall further assume throughout this paragraph that
\begin{itemize}
\item $c\geq 1$,
\item $\eps>0$ being fixed, the sequence  of minimizers $\{\gamma_{p, \eps}\}_{p\geq 1}$ weakly star converges as $p\to \infty$ to some $\gamma_{\infty}$, with ($\infty$-cyclically monotone) support $\Gamma_\infty$.
\end{itemize}
Let us define for every $(x,y)\in \R^d\times \R^d$ 
	\[I_{\infty}(x,y):=\sup_{k\ge2}\sup_{(x_i,y_i)_{i=2}^k \subset \Gamma_\infty}\sup_{\sigma\in\Sigma(k)} \Big\{\ \max_{1\le i\le k}\{c(x_i,y_i)\}-\max_{1\le i\le k}\{c(x_i,y_{\sigma(i)})\} \Big\},  \]
	where $(x_1,y_1)=(x,y)$. 	Also define 
	\[\tI_{\infty}(x,y):=\sup_{k\ge2}\sup_{(x_i,y_i)_{i=2}^k \subset \Gamma_\infty}  \Big\{\ \max_{1\le i\le k}\{c(x_i,y_i)\}-\max_{1\le i\le k}\{c(x_i,y_{i+1})\} \Big\},\]
	where $(x_1,y_1)=(x,y)$ and $y_{k+1}=y_1$.
In our supremal optimal transport setting, we cannot really expect that $I_{\infty}$ is a good rate function for $\{\gamma_{p, \eps}\}_{p\geq 1}$; indeed, $\argmin_{\Pi(\mu,\nu)}J_{\infty}$ is  unchanged when replacing $c$ with a strictly increasing function of $c$, while the same does not hold for the function $I_{\infty}$. However it can be interesting to have a better understanding of the function $I_\infty$, which still provides an upper bound for the family $\{\gamma_{p,\eps}\}$ (see Proposition \ref{lem2:upperbound}). 
\begin{lem}\label{propiinfty}
Let $I_\infty$ and $\tI_\infty$ be defined as above, then
\begin{itemize}
\item $I_\infty$ and $\tI_{\infty}$ are related by $I_{\infty}=\max(0, \tI_{\infty})$,
\item  $I_\infty$ and $\tI_{\infty}$ are lower semicontinuous, $I_\infty\geq 0$, $I_{\infty}=0$ on $\Gamma_\infty$,
\item $I_\infty$ and $\tI_{\infty}$ coincide on $(\spt(\mu)\times \R^d) \cup (\R^d \times \spt(\nu))$.
\end{itemize}
\end{lem}
\begin{proof}
The fact that $I_\infty \geq \max(0, \tI_{\infty})$ is obvious as well as the fact that $\tI_{\infty}= 0$ on $\Gamma_\infty$. 
\\We now prove the converse inequality. Fix now $(x,y)=(x_1,y_1) \in \R^d \times \R^d$, $k\geq 2$, $(x_2, y_2), \ldots (x_k, y_k)$ in $\Gamma_{\infty}$ and $\sigma \in \Sigma(k)$. We can then partition $\{1, \ldots, k\}$ into $I_0$ the (possibly empty) set of fixed-points of $\sigma$ and disjoint (empty if $\sigma$ is the identity) orbits $I_1, \ldots, I_l$ on each of which $\sigma$ is a cycle, this means that for $j=1, \ldots, l$,  we may denote $(x_i, y_i)_{i\in I_j}$ as $(\tx^j_r,\ty^j_r)_{r=1, \ldots, \vert I_j\vert}$ and  $(x_i, y_{\sigma(i)})_{i\in I_j}$ as $(\tx^j_r,\ty^j_{r+1})_{r=1, \ldots, \vert I_j\vert}$  with the convention $\ty^j_{\vert I_j \vert+1}=\ty^j_{1}$. We now observe that 
\[\max_{1\le i\le k}\{c(x_i,y_i)\}-\max_{1\le i\le k}\{c(x_i,y_{\sigma(i)})\} \leq \max_j \Big\{\max_{i\in I_j} c(x_i, y_i)- \max_{i\in I_j}  c(x_i, y_\sigma(i)) \Big\}.  \]
where the max with respect to $j$ is taken on indices for which $I_j$ is nonempty.
  To shorten notations, for such a $j$ let us set 
 \[\beta_j:=\max_{i\in I_j} c(x_i, y_i)- \max_{i\in I_j}  c(x_i, y_\sigma(i)).\] 
 Of course if $I_0$ is nonempty, $\beta_0=0$, now if $j\geq 1$ and $I_j$ is nonempty
 \[\beta_j =\max_{r=1, \ldots, \vert I_j\vert} c(\tx^j_r, \ty^j_r)- \max_{r=1, \ldots, \vert I_j\vert} c(\tx^j_r, \ty^j_{r+1})\leq \tI_{\infty}(\tx_1^j, \ty_1^j).\]
 So, if $(\tx_1^j, \ty_1^j)=(x_1, y_1)$, $\beta_j \leq \tI_{\infty}(x,y)$ and if $(\tx_1^j, \ty_1^j)\neq (x_1, y_1)$, then $(\tx_1^j, \ty_1^j)\in \Gamma_{\infty}$, hence $\tI_{\infty}(\tx_1^j, \ty_1^j)= 0$ by the definition of $\tI_\infty$ and the fact that $\Gamma_\infty$ is $\infty$-cyclically monotone. In other words, we can bound from above each $\beta_j$ by $\max(0,\tI_\infty(x,y))$. Taking suprema with respect to $k$,  $(x_2, y_2), \ldots (x_k, y_k)$ in $\Gamma_{\infty}$ and $\sigma \in \Sigma(k)$, we thus get  $I_{\infty}\leq \max(0, \tI_{\infty})$. Moreover, since $\tI_\infty \leq 0$ on $\Gamma_\infty$, $I_{\infty}= \max(0, \tI_{\infty})=0$ on $\Gamma_{\infty}$
\smallskip
\\Lower semi continuity of $I_\infty$ and $\tI_{\infty}$ follows from the continuity of $c$. Finally assume that $x\in \spt(\mu)$ and $y\in \R^d$, since $\Gamma_\infty=\spt(\gamma_\infty)$ is compact and $\gamma_\infty\in \Pi(\mu, \nu)$, there exists $y'\in \R^d$ such that $(x,y')\in \Gamma_\infty$. Taking $(x_1,y_1)=(x,y)$, $(x_2, y_2)=(x,y')$ as a competitor in the definition of $\tI_{\infty}(x,y)$ we see that $\tI_{\infty}(x,y)\geq 0$ hence $I_{\infty}(x,y)=\tI_{\infty}(x,y)$. The same argument shows that $I_\infty$ and $\tI_\infty$ coincide on $\R^d\times \spt(\nu)$. 
\end{proof}

\begin{lem}\label{lem:uppbound}
	Let us fix $(x,y)\in \R^d\times \R^d$. Suppose that for some $\delta\in \R$, $k\in \N$, $k\geq 2$ and  $(x_i,y_i)_{i=2}^k\subset \spt\gamma_{\infty}$,  we have
	\[\max_{1\le i\le k}\{c(x_i,y_i)\}-\max_{1\le i\le k}\{c(x_i,y_{i+1})\}>\delta, \ \text{where }(x_1,y_1):=(x,y). \]
	Then there exist $\alpha>0$, $r>0$ and $p_0\geq 1$ such that 
	\[\gamma_{p,\eps}(B_r(x,y))\le\alpha e^{\frac{-p\delta}{\eps}}, \ \forall p\geq p_0,  \]
	where $\gamma_{p,\eps}$ is the minimizer of $J_{p, \eps}$.
\end{lem}
\begin{proof}
Of course if $\delta \leq 0$, one can just take $\alpha =1$ so we may assume that $\delta>0$. Reasoning as in  the proof of Theorem \ref{th:inftymon} (recall that we have assumed $c\geq 1$), we know that there exist $p_0$ and $r>0$ such that 
	\[\sum_{i=1}^{k}c^p(x'_i,y'_i)-\sum_{i=1}^kc^p(x'_i,y'_{i+1})> p\delta,  \]
	for every $p\geq p_0$ and $(x'_i,y'_i)_{i=1}^k\subset B_r(x_1,y_1)\times\cdots\times B_r(x_k,y_k)$. 
	Then $B_r(x_1,y_1)\times\cdots\times B_r(x_k,y_k)\subset A_{k, c^p}(p\delta)$ so, thanks to Lemma \ref{nutzlemma},
	\[\gamma^k_{p,\eps}(B_r(x_1,y_1)\times\cdots\times B_r(x_k,y_k))\le e^{-\frac{p\delta}{\eps}}.\]
	Moreover $\liminf_{p\to\infty}\gamma_{p,\eps}(B_r(x_i,y_i))\ge\gamma_{\infty}(B_r(x_i,y_i))>\beta$, for all $2\le i\le k$, for some $\beta>0$ since $(x_i,y_i)_{i=2}^k\subset\spt\gamma_{\infty}, $ then
	\[\gamma_{p,\eps}(B_r(x,y))\le \left(\frac{\beta}{2}\right)^{1-k}e^{-\frac{p\delta}{\eps}}, \]
	for all $p \geq p_0$ (possibly replacing $p_0$ with a larger one).
\end{proof}
\begin{prop}\label{lem2:upperbound}
Under the assumptions of this paragraph,  for any compact set $C\subset \R^d\times \R^d$, one has
	\[\limsup_{p\to\infty}\frac{\eps}{p}\log\gamma_{p,\eps}(C)\le -\inf_{ C \cap (\spt(\mu)\times \spt(\nu))} \tI_\infty \le   -\inf_{ C}I_{\infty}. \]
	\end{prop}
	\begin{proof}
	First note that since $\gamma_{p, \eps}$ is supported on $\spt(\mu)\times \spt(\nu)$, 
	\[\gamma_{p,\eps}(C)=\gamma_{p,\eps}(C \cap (\spt(\mu)\times \spt(\nu)))\]
	and there is noting to prove if $C$ is disjoint from
	$\spt(\mu)\times \spt(\nu)$. Therefore we can assume that $C\cap( \spt(\mu)\times \spt(\nu)) \neq \emptyset$. It then follows from Lemma \ref{propiinfty} that
	\[\inf_{ C \cap (\spt(\mu)\times \spt(\nu))} \tI_\infty  = \inf_{ C \cap (\spt(\mu)\times \spt(\nu))} I_\infty  \geq \inf_{ C} I_\infty.\]
	Now let  $\eta>0$ and $(x,y)\in C\cap (\spt(\mu)\times \spt(\nu))$. By definition of $\tI_\infty(x,y)$ there exist $k\ge 2$ and $(x_i,y_i)_{i=2}^k\subset \Gamma_\infty$, such that (setting as usual $(x_1,y_1)=(x,y)$ and $y_{k+1}=y$)
		\[\max_{1\le i\le k}\{c(x_i,y_i)\}-\max_{1\le i\le k}\{c(x_i,y_{i+1})\}> \min( \eta^{-1}, \tI_{\infty}(x,y))-\eta.
		\]
		Note that the truncation is used to handle the case where $\tI_\infty(x,y)=+\infty$. 
		By Lemma \ref{lem:uppbound} we know that there exist $\alpha,r>0$ such that 
		\[\gamma_{p,\eps}(B_r(x,y))\le \alpha\exp\left(\frac{-p( \min( \eta^{-1}, \tI_{\infty}(x,y))-\eta)}{\eps}\right).\]
		Then 
		\[\limsup_{p\to\infty}\frac{\eps}{p}\log\gamma_{p,\eps}(B_r(x,y))\le - \min( \eta^{-1}, \tI_{\infty}(x,y))+\eta  \]
		and, by compactness of $C$,
		\[\limsup_{p\to\infty}\frac{\eps}{p}\log\gamma_{p,\eps}(C)\le -\inf_{C \cap (\spt(\mu)\times \spt(\nu))  } \min(\eta^{-1}, \tI_{\infty})+\eta \]
		which, letting $\eta\to 0^+$, yields the desired upper bound. 
	\end{proof}

\section{Numerical results}\label{sec-num}
In this section, we present several  numerical examples, with the aim of illustrating the discussions and theoretical analysis of the previous sections. We shall consider discrete marginals; let $N,M\in\N$, with a slight abuse of notation, we will denote by $\mu$ and $\nu$ both the measures and the vectors of weights  $(\mu_i)_{i=1}^N$ and $(\nu_j)_{j=1}^M$ and  $\gamma$ will denote both the transport plan and the $N\times M$ matrix $(\gamma^{ij})$.  For fixed $p,\eps>0$, in this discrete setting, the minimization of   $J_{p,\eps}$ reads
\begin{equation}\label{pepsDEOT}
	\min_{\Pi(\mu,\nu)}	\left(\sum_{i,j}\gamma^{ij} c^p_{ij}+\eps   \sum_{ij} \gamma^{ij} \log \Big( \frac{\gamma^{ij}}{\mu_i \nu_j }\Big)   \right)^{\frac1p}.
\end{equation}
Raising the above cost to the power $p$, which does not change the minimizer, leads to  a standard entropic transport problem. For such problems, we used in all our examples Sinkhorn's algorithm (see for instance Chapter 4 in \cite{CP19}) to find a good approximation (with error smaller than $10^{-5}$) of the solution.

\smallskip

If $v_\infty\ge 1$, in light of Theorem \ref{th:gammaconvergence}, we expect the output $\gamma$ of the Sinkhorn algorithm to be, for suitable $p$ and $\eps$, also a good approximation of an optimal plan for the discretized $L^\infty$- optimal transport problem
\[v_{\infty}:=\min_{\gamma\in \Pi(\mu,\nu)}\max_{i,j}\left\{ c_{i,j}\, : \, \gamma^{ij}\neq 0 \right\}.\]
Furthermore, if $c\ge1$, thanks to Theorem \ref{th:inftymon}, we expect to find a plan close to an $\infty$-cyclically monotone one.

\begin{rem}\label{rem:vinftybyhands}
As the set of transport plans $\Pi(\mu, \nu)$ is a convex polytope, for any $\gamma\in \Pi(\mu,\nu)$ there exists a finite set of indices $S$, such that $\gamma=\sum_{s\in S} a_s\gamma_s$, with $a_s > 0$, $\sum a_s=1$ and $\gamma_s$ an extreme point of $\Pi(\mu, \nu)$. If $N=M$ and $\mu_i=\nu_j=\frac1N$, the set $\Pi(\mu,\nu)$ is the set of the so-called \textit{bi-stochastic} matrices, whose extreme points, by Birkhoff's theorem, form the set of  pemutation matrices. 
We observe that, by definition of $\gamma-\ess$, $J_\infty(\gamma)= \max_{s\in S} J_{\infty}(\gamma_s)$ and thus the minimum of $J_\infty$ is attained at some permutation matrix. Therefore, if $N=M$ and $\mu_i=\nu_j=\frac1N$
\begin{equation*}
	v_\infty=\min_{\sigma\in\Sigma(N)}\max_{i}c_{i,\sigma(i)}.
\end{equation*}
This can be in principle used to compute $v_\infty$ exactly. However this is not particularly useful in practice; regarding for instance the example on bottom of Figure \ref{fig:composition_examples_4}, even if the size of $\mu$ and $\nu$ is the same, in order to calculate the exact value of $v_\infty$ we should be able to perform $100!$ evaluations, which is infeasible in practice!
 \end{rem}

 All the examples in this section, will be in dimension $d=2$, $\mu$ will be represented by blue points, $\nu$ by red points and the plan will be represented by arrows:
the black ones indicate that a blue point is sent to a red point with high probability, while the gray ones indicate that a blue point is sent to a red point with lower probability (but still not negligible). 
 
\begin{figure}[H]
	\centering
	\includegraphics[scale=0.3]{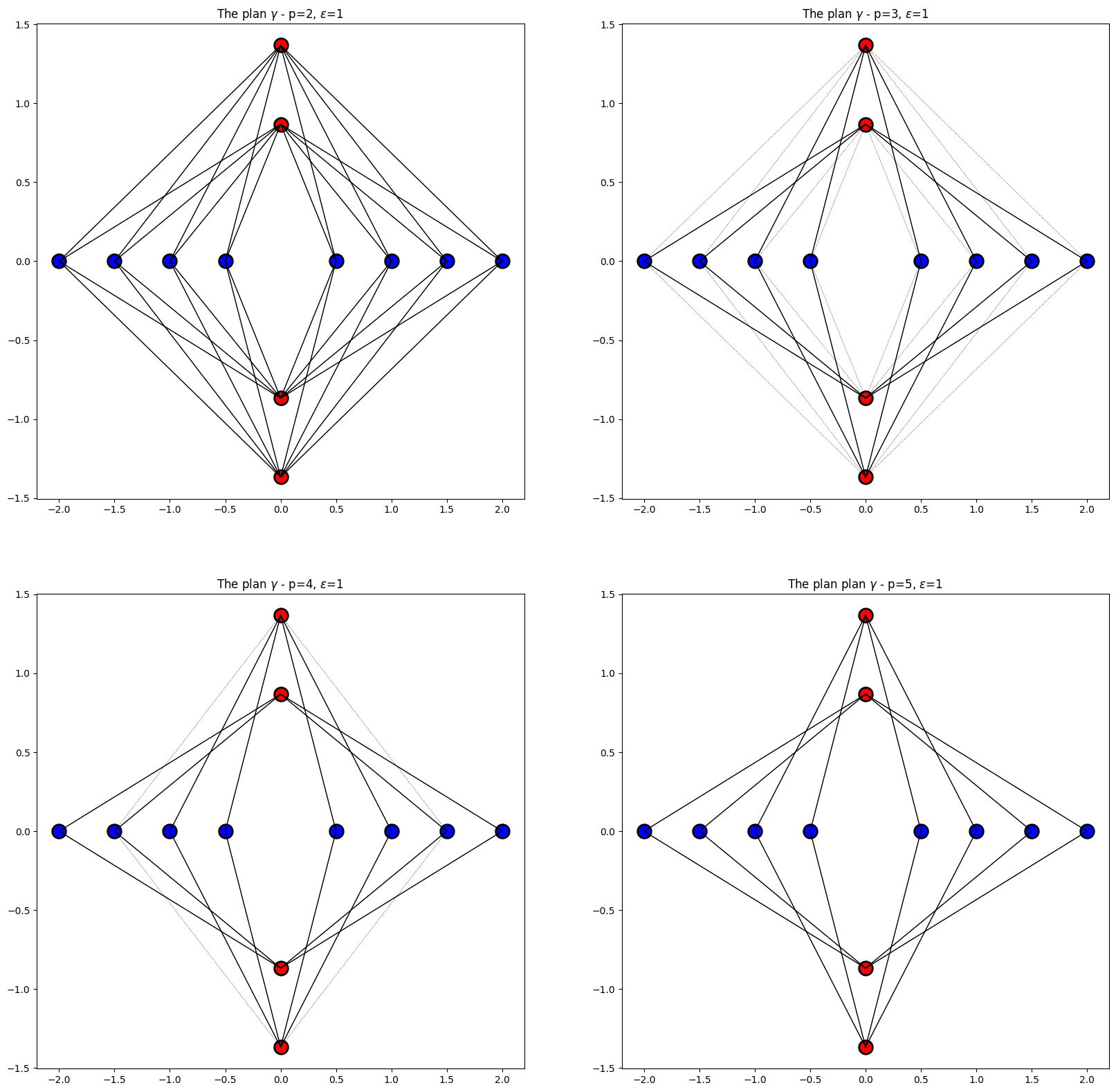}
	\caption{Example of convergence of the plan to the $\infty$-cm plan: $c(x,y)=|x-y|^p$, $p\in\{2,3,4,5\}$, $\eps=1$ and $\mu$ and $\nu$ having orthogonal supports.}
	\label{fig:exdiscrete}
\end{figure}
\noindent In the first example, as shown by Figure \ref{fig:exdiscrete}, we consider $c^p=|x-y|^p$, for $p\in\{2,3,4,5\}$, $\mu$ which is uniformly concentrated on the blue points \[\{(-2,0), (-1.5,0), (-1,0), (-0.5,0),(0.5,0), (1,0), (1.5,0), (2,0)\}\] and $\nu$ on the red points \[\{(0,-1.367), (0,-0.867), (0,867), (0,1.367)\}.\]
Note that with this choice of $\spt\mu$ and $\spt\nu$, $c\ge1$ everywhere and therefore, thanks to Theorem \ref{th:gammaconvergence} and Theorem \ref{th:inftymon}, $\Gamma$-convergence and convergence of the outputs towards $\infty$-cm plans still hold choosing $\eps=1$.
We observe that for $p=2$, every transport plan $\gamma$ is optimal. Indeed, by the orthogonality of the two supports, any plan is concentrated on a cyclically monotone set (see \eqref{ccmon}) and, as recalled in Section \ref{ldub} (see for instance \cite{GangboM,Santambrogio}), this is a sufficient  optimality condition. Here, since we look for a plan which minimizes the regularized problem which involves the entropy, the Sinkhorn algorithm selects the most diffuse one, as evidenced by the picture on the upper left of Figure \ref{fig:exdiscrete}. The other three pictures in Figure \ref{fig:exdiscrete} show that convergence towards an $\infty$-cm plan is really fast and it occurs already for $p=5$. 
\begin{figure}[H]
	\centering
	\includegraphics[scale=0.28]{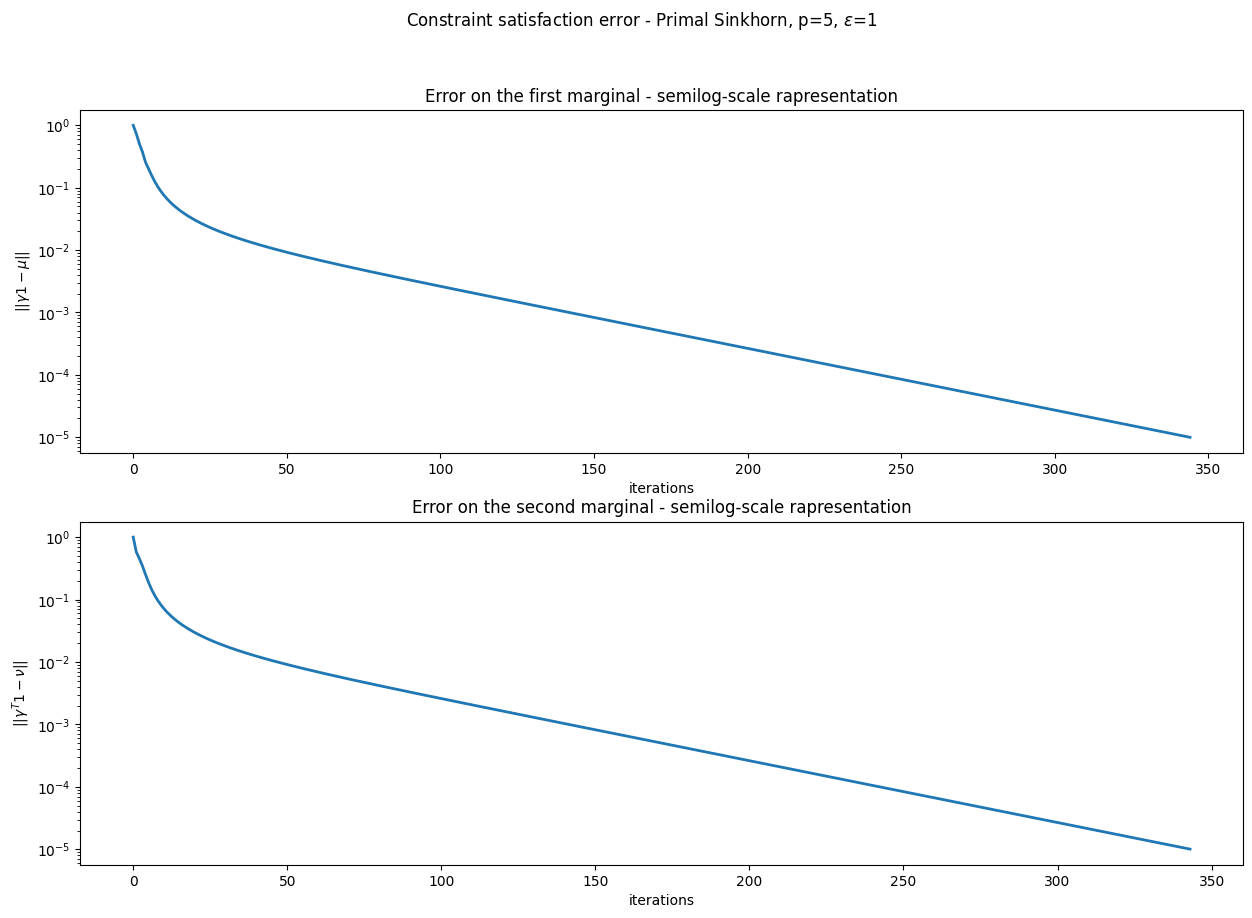}
	\caption{Error on the marginals: the first image shows the error $|\gamma\mathbbm{1}_4-\mu|$ of the output $\gamma$ on the first marginal and the second one the error $|\gamma^\intercal\mathbbm{1}_8-\nu|$ on the second marginal.}
	\label{fig:exdiscrete_errormarginals}
\end{figure}
\noindent Regarding the accuracy, Figure \ref{fig:exdiscrete_errormarginals} shows that for $p=5$ and $\eps=1$ the distance $|\gamma\mathbbm{1}_4-\mu|$ between the first marginal of the output $\gamma$ and the distance $|\gamma^\intercal\mathbbm{1}_8-\nu|$ between the second marginal of $\gamma$ and $\nu$ is of the order of $10^{-5}$ after only $350$ iterations.\\
\begin{figure}[H]
	\centering
	\includegraphics[scale=0.28]{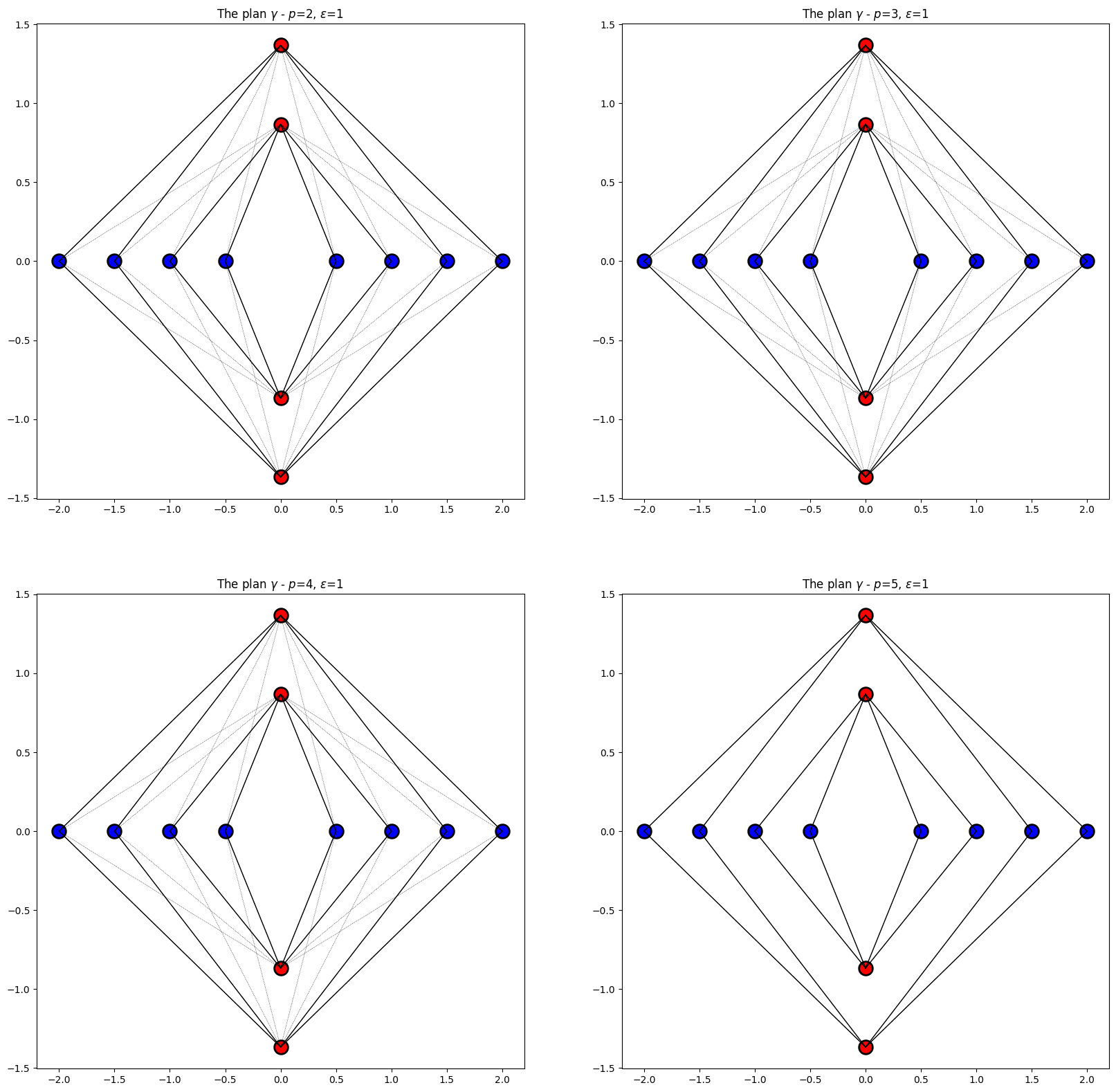}
	\caption{Example of convergence of the plan to the $\infty$-cm plan for $c(x,y)=\left(\max\{|x_1-y_1|,|x_2-y_2|\}\right)^p$, for $p\in\{2,3,4,5\}$, $\eps=1$ and $\mu$ and $\nu$ having orthogonal supports.}
	\label{fig:exdiscrete_costmax}
\end{figure}
\noindent We have also considered the same example (see Figure \ref{fig:exdiscrete_costmax}) with the cost function $c^p(x,y):=\left(\max\{|x_1-y_1|,|x_2-y_2|\}\right)^p$. In this case the convergence is still fast and the error is small after few iterations (of order $10^{-5}$ after about $180$ iterations).

\begin{rem}\label{rmk:aboutthesizeofc}
	When $c>1$, on the one hand,   we don't need $\eps$ to be small and we can even  take it large as $p$ grows (by case 2. in Theorem \ref{th:gammaconvergence} we can even choose for instance $\eps_p=(1+\lambda)^p$). On the other hand, we can encounter some difficulties when computing the Gibbs kernel $K_{ij}=e^{-\frac{c^p_{ij}}{\eps}}$: if $p$ is large it can happen that, for some $i,j$, $K_{i,j}=0$ making impossible to perform the division in the iterations of the primal version of the Sinkhorn algorithm. Fortunately, this problem can be overcome using the Log-Domain version (see for instance Section 4.4 in \cite{CP19}), as we did in the following example, represented by Figure \ref{fig:composition_examples_4}.
\end{rem}
Figure \ref{fig:composition_examples_4}, which shows a comparison among three different examples, considered for $p=2$ on the left and for $p=15$ on the right and $\eps=1$.
The two pictures on top in Figure \ref{fig:composition_examples_4} show the representation by arrows of the output when $\mu$ is uniformly concentrated on $400$ points which discretize the unitary square and $\nu$ is uniformly concentrated on the points $(1,2)$ and $(2,1)$. This is a discretization of the case $\mu$ uniform on the square $[0,1]^2$, where (see also Example 2.2 in \cite{CDPJ}) every $\gamma\in \Pi(\mu,\nu)$ is optimal for the problem 
\[\inf_{\gamma\in \Pi(\mu, \nu)} \gamma-\esup  c=\Vert c \Vert_{L^{\infty}(\gamma)}.\]
Indeed \begin{align*}
	\Vert c \Vert_{L^{\infty}(\gamma)}&=\sup\{c(x,y) \, : \, (x,y)\in\spt\gamma  \}\\&=|(0,0)-(1,2)|=|(0,0)-(2,1)|=\sqrt{5},
\end{align*}
for every $\gamma\in \Pi(\mu,\nu).$ Since every plan is optimal, when $p$ is smaller, as shown in the picture on the left, the role of the entropy is more important and the algorithm selects the most diffuse plan. While increasing the value of $p$ the entropy becomes more and more negligible and output becomes sparser: already for $p=15$ (on the right) the output is a good approximation of the $\infty$-cyclically monotone plan, which in this case is unique (see Theorem 5.6 in \cite{CDPJ}). A small variation, represented by the two figures in the middle, is to consider $\nu$ which is not uniformly concentrated on the points $(1,2)$ and $(2,1)$. Here we have taken $\nu=0.1\delta_{(1,2)}+0.9\delta_{(2,1)}$. Finally, on the bottom, we have implemented the case in which also $\nu$ is the discretization of an absolutely continuous measure. Here $\mu$ approximates the square $[-0.25,0.25]\times[-0.25,0.25]$ and $\nu$ the rectangle $[1.25,1.5]\times[-0.5,0.5]$ and both measures are supported on $100$ points. As previously, one can  notice that for $p=2$ the entropy plays an important role and the algorithm selects the most diffuse plan, while, already for $p=15$ the plan is considerably  sparser.

\begin{figure}[H]
	\centering
	\includegraphics[scale=0.3]{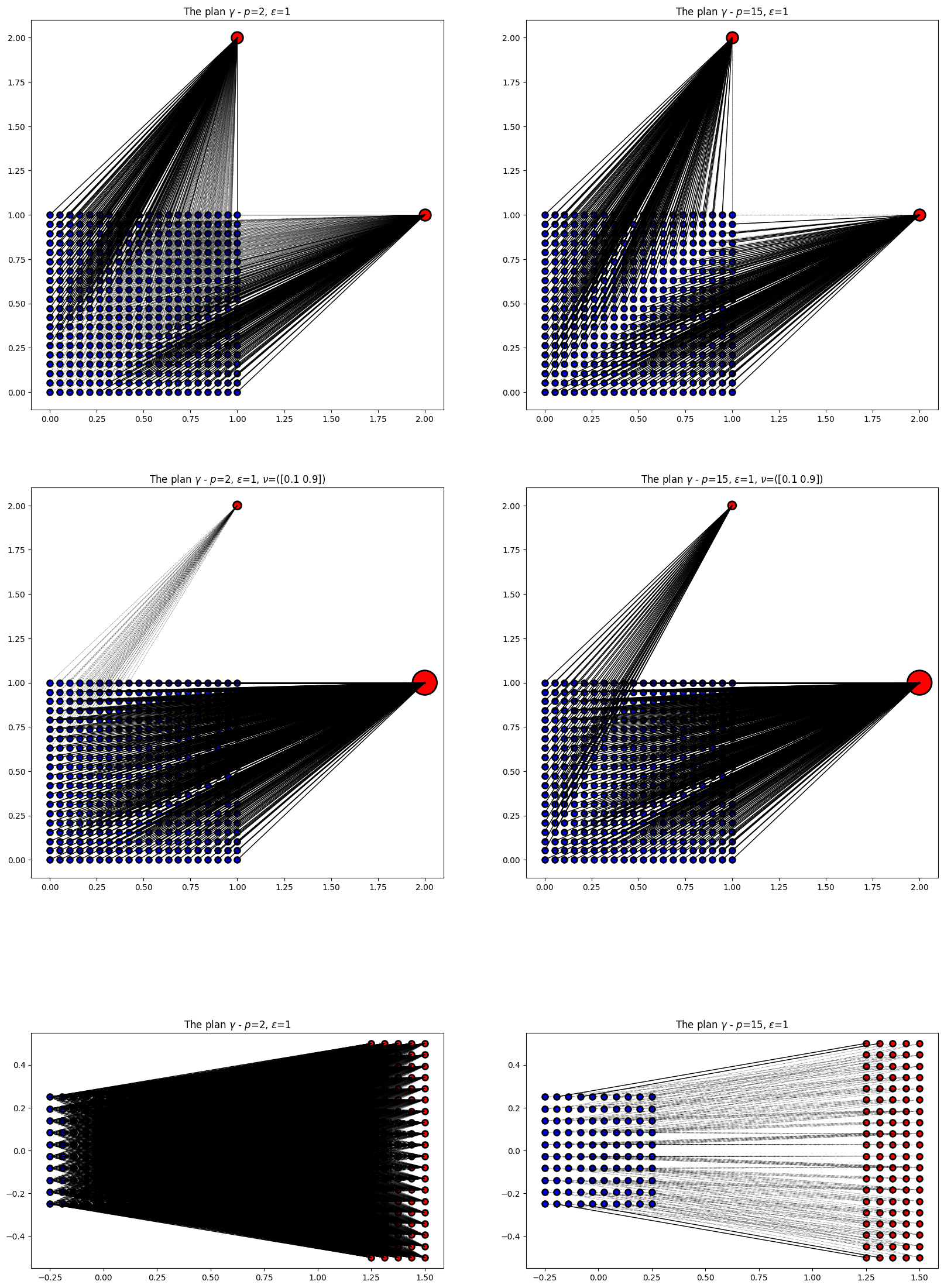}
	\caption{Comparison among three different examples: $\eps=1$, $c(x,y)=|x-y|^p$, on the left $p=2$, on the right $p=15$.  On top: $\mu$ a uniform discretization of the unitary square and $\nu$  uniformly concentrated on the points $(1,2)$ and $(2,1)$. In the middle: $\mu$ the same discretization of the unitary square, $\nu=0.1\delta_{(1,2)}+0.9\delta_{(2,1)}$ (the point $(2,1)$ is represented by a bigger dot). On bottom: $\mu$ a uniform discretization of the square $[-0.25,0.25]\times[-0.25,0.25]$ and $\nu$ of the rectangle $[1.25,1.5]\times[-0.5,0.5]$.  }
	\label{fig:composition_examples_4}
\end{figure}

We are now interested in the asymptotic behavior of $v_p:=\min_{\Pi(\mu, \nu)}J_p$ and we want to numerically represent the upper and lower bounds on the speed of convergence of $v_p$ towards $v_\infty:=\min_{\Pi(\mu, \nu)}J_\infty$ proved in Proposition \ref{prop:upperbound} and Proposition \ref{prop:lowerbound}. In order to apply Proposition \ref{prop:upperbound} and Proposition \ref{prop:lowerbound} it is enough to assume a lower bound on $v_\infty$ and not a pointwise one on $c$.

\begin{figure}[H]
	\centering
	\includegraphics[scale=0.3]{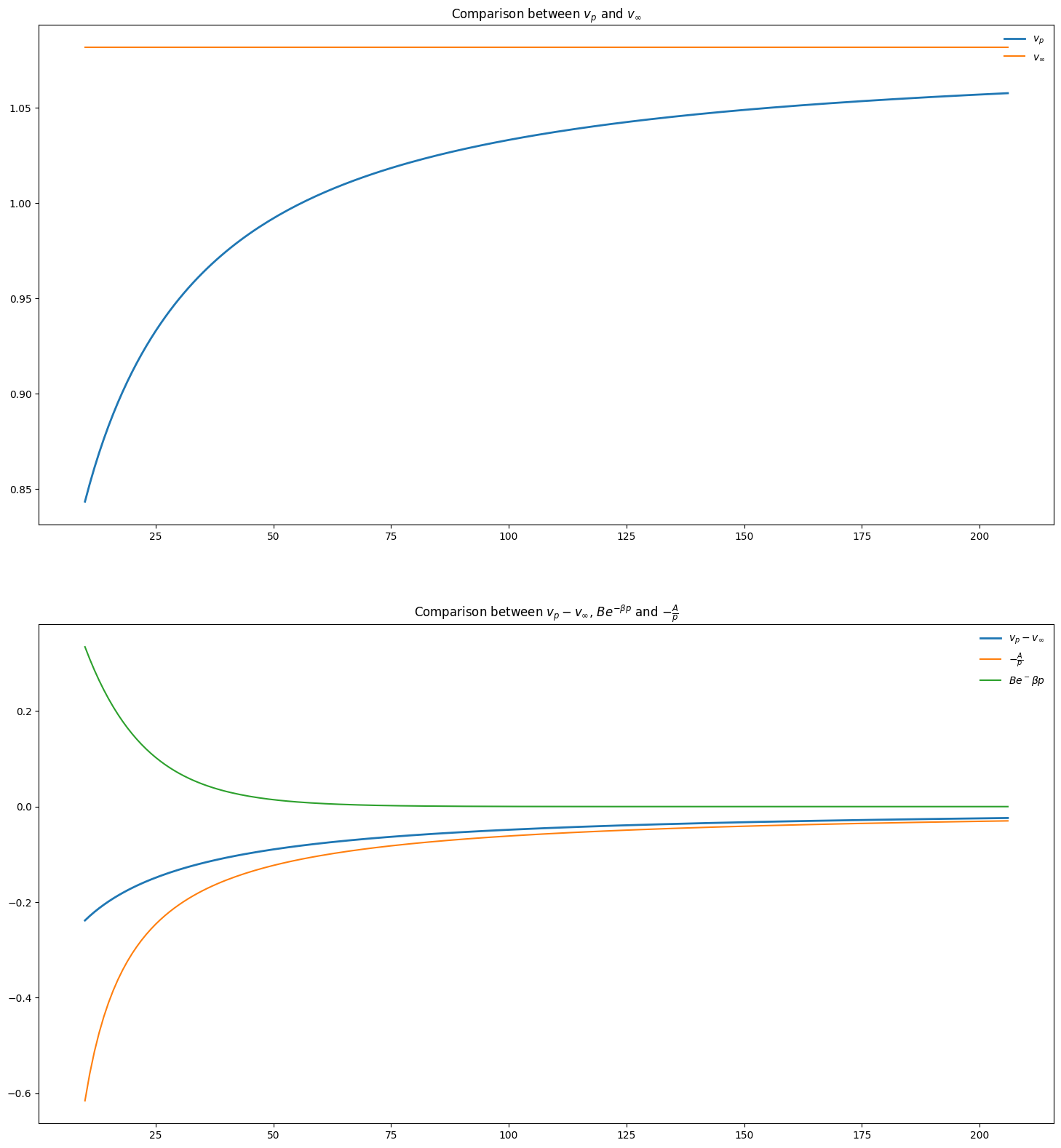}
	\caption{Comparison among the speed of convergence of $v_p-v_\infty$, $Be^{-\beta p}$ and $-\frac{A}{p}$ for $p\in [10,206]$, $\mu$ and $\nu$ as the ones on top of Figure \ref{fig:composition_examples_4}. On top: $v_p$ in blue and $v_\infty$ in orange. On bottom: $Be^{-\beta p}$ in green, $-\frac{A}{p}$ in orange and $v_p-v_\infty$ in blue. Here $A,B$ are obtained by linear regression (least squares)  and $\beta=\log(v_\infty)$ (the same $\beta$ as in Proposition \ref{prop:upperbound}).}
	\label{fig:exChaDepJuu_comparisonconvergence}
\end{figure}
\noindent Figure \ref{fig:exChaDepJuu_comparisonconvergence} provides an example of the asymptotic behavior of $v_p$ and of the speed of convergence in the case of $\mu$ and $\nu$ as the ones represented in the two pictures on top of Figure \ref{fig:composition_examples_4}. In light of what we have just remarked, we have re-scaled the cost $c$ 
in order to have $v_\infty\simeq1.08166$. For $p\in [10,206]$ the image on top of Figure \ref{fig:exChaDepJuu_comparisonconvergence} shows in blue how $v_p$ changes varying $p$, while $v_\infty$ is constant and is represented by the orange line. On bottom of Figure \ref{fig:exChaDepJuu_comparisonconvergence} we have represented in blue $v_p-v_\infty$, in green the upper bound $Be^{-\beta p}$ and in orange the lower bound $-\frac A p$, where $\beta=\log(v_\infty)$ by Proposition \ref{prop:upperbound} (indeed in this case $c$ is Lipschitz so  $\alpha=1>\log(v_\infty)$) and $A, B$ have been estimated by a linear regression method (by least squares).

Finally,  an example in which it is possible (even if it is really slow!) to compute $v_\infty$ exactly (see Remark \ref{rem:vinftybyhands})
is represented in Figure \ref{fig:exvinfbyhands}. 
Here $\mu$ is concentrated on $8$ points, given by \[\left\{(x_1,x_2) \, : \, x_1=-0.25+ 0.125\cdot i,\, i=1,\dots, 4, \ x_2\in\{-0.1,0.1\} \right\}\] and $\nu$ is concentrated on $8$ equidistant points of the segment starting from the point $(0.625,1.25)$ to the point $(1.25,0)$  of the line $y_2=-2y_1+2.5$.
We have computed $v_\infty$ for the cost $c(x,y)=|x-y|$ applying Remark
\ref{rem:vinftybyhands}, and we have obtained that $v_\infty \simeq 1.38647347$   and that the points which are at the minimal-maximal distance are $x_*=(-0.25, -0.1)$ and $y_*=(0.98214286, 0.53571429)$, connected by the purple segment in the picture. Regarding the speed of convergence we rescaled the cost in order to decrease further $v_\infty\simeq 1.052460609$. As shown in Figure \ref{fig:exvinfbyhands_comparisonconvergence}, $v_p$ is calculated varying $p$ in the interval $[10,172]$, with $\eps=500^2$. We observe that in this case, as shown in the picture on top, $v_p$ is initially smaller than $v_\infty$, then it increases becoming greater and finally it starts decreasing converging to $v_\infty$.

\begin{figure}[H]
	\centering
	\includegraphics[scale=0.3]{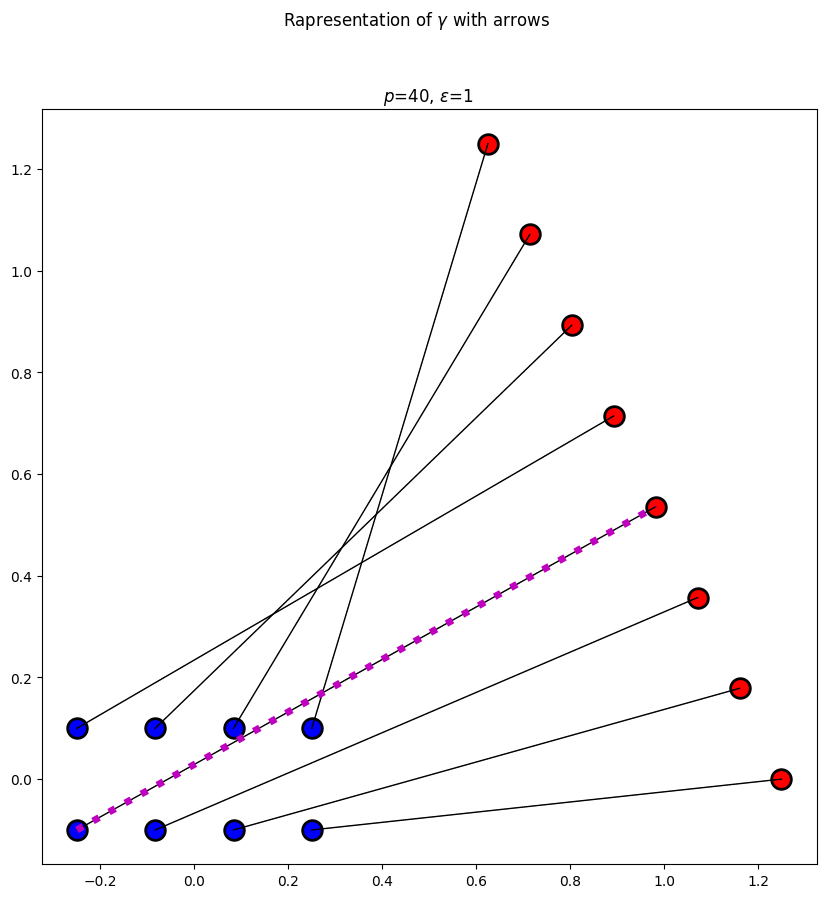}
	\caption{$\mu$ and $\nu$ uniformly distributed both concentrated on $8$ points. The value of $v_\infty$ is about $1.38647347$ and it is obtained transporting mass between the two points connected by the yellow segment.}
	\label{fig:exvinfbyhands}
\end{figure}
\begin{figure}[H]
	\centering
	\includegraphics[scale=0.3]{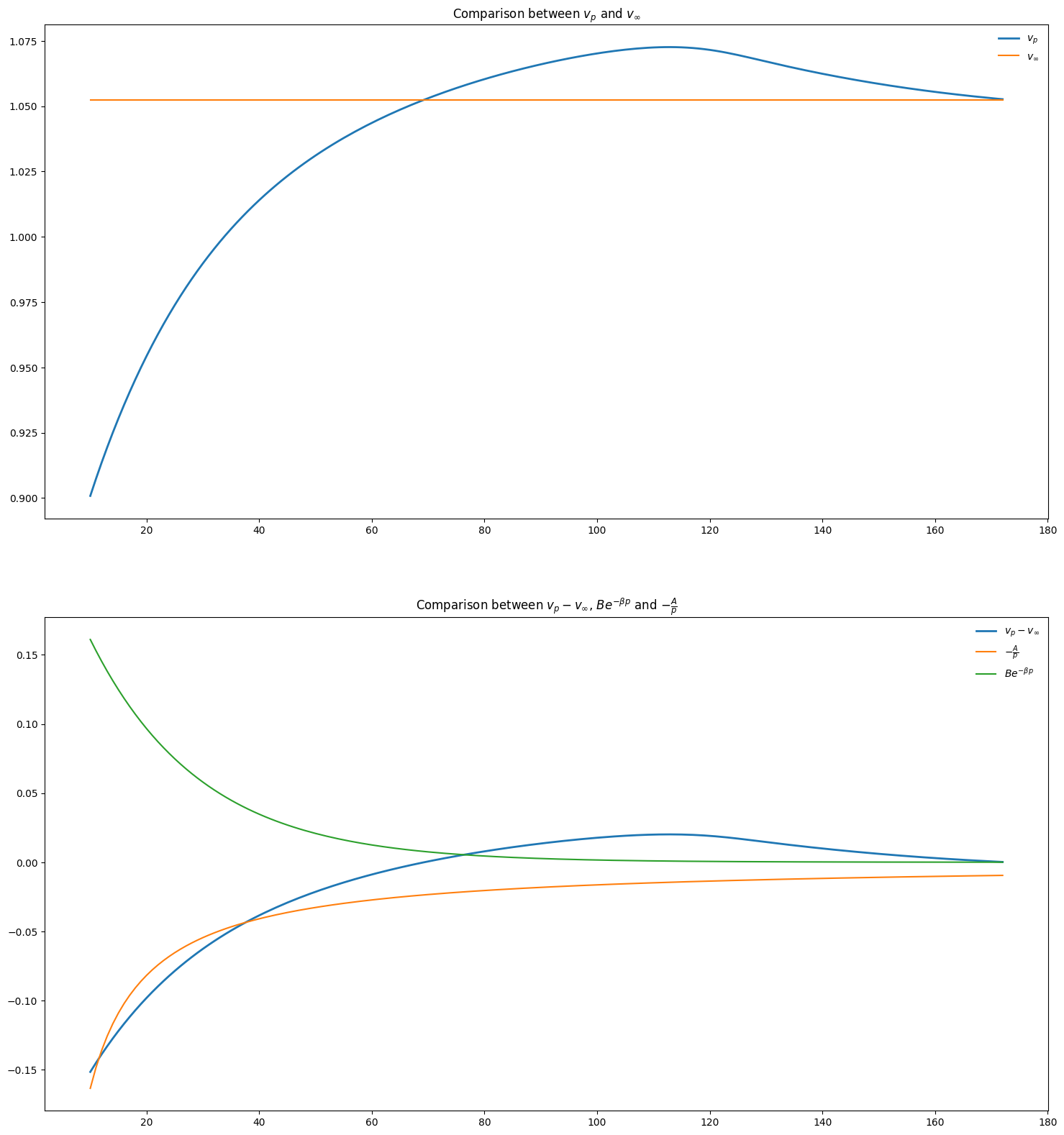}
	\caption{Comparison among the speed of convergence of $v_p-v_\infty$, $Be^{-\beta p}$ and $-\frac{A}{p}$ for $p\in[10,172]$ and $\eps=500^2$, $\mu$ and $\nu$ as the ones in Figure \ref{fig:exvinfbyhands}. On top: $v_p$ in blue and $v_\infty$ in orange. On bottom: $Be^{-\beta p}$ in green, $-\frac{A}{p}$ in orange and $v_p-v_\infty$ in blue. Here $A,B$ are obtained by linear regression (least squares) and $\beta=\log (v_\infty)$ (the same $\beta$ as in Proposition \ref{prop:upperbound}).}
	\label{fig:exvinfbyhands_comparisonconvergence}
\end{figure}
 
\noindent{{\bf{Acknowledgments:}} G.C. acknowledges the support of the Lagrange Mathematics and Computing Research Center. The research of C.B. and L.DP is partially financed  by the {\it ``Fondi di ricerca di ateneo, ex 60 $\%$''}  of the  University of Firenze and is part of the project  {\it "Alcuni problemi di trasporto ottimo ed applicazioni"}  of the  GNAMPA-INDAM. C.B.  and G.C. also acknowledge the support of the French Agence Nationale de la Recherche through the project MAGA (ANR-16-CE40-0014).}    
 
\bibliographystyle{plain}

\bibliography{biblio}

\end{document}